\documentclass[a4paper]{amsart}
\title[A representation formula for indefinite improper affine spheres]{
 A representation formula for indefinite improper affine spheres
}
\date{December 7, 2007}
\keywords{improper affine spheres, singularities, 
	  cuspidal cross cap}
\usepackage[dvips]{graphicx} 
\usepackage{verbatim,enumerate}
\usepackage{amssymb}
\usepackage[usenames]{color}
\usepackage{amsthm}
\theoremstyle{plain}
 \newtheorem{theorem}{Theorem}[section]
 \newtheorem*{theorem*}{Theorem}

 \newtheorem*{lemma*}{Lemma}
 \newtheorem{proposition}[theorem]{Proposition}
 \newtheorem{fact}[theorem]{Fact}
 \newtheorem*{fact*}{Fact}

\theoremstyle{remark}
 \newtheorem{definition}[theorem]{Definition}
 \newtheorem{remark}[theorem]{Remark}
 \newtheorem*{remark*}{Remark}
 \newtheorem*{acknowledgements}{Acknowledgements}
 \newtheorem{example}[theorem]{Example}
\numberwithin{equation}{section}
\renewcommand{\theenumi}{{\rm(\arabic{enumi})}}
\renewcommand{\labelenumi}{\theenumi}

\newcommand{\R}{\boldsymbol{R}}
\newcommand{\C}{\boldsymbol{C}}

\newcommand{\Ker}{\operatorname{Ker}}

\newcommand{\vol}{\operatorname{vol}}

\renewcommand{\Re}{\operatorname{Re}}

\renewcommand{\phi}{\varphi}
\renewcommand{\epsilon}{\varepsilon}

\author{Daisuke Nakajo}
\address[Daisuke Nakajo]{%
   Graduate School of Mathematics,
   Kyushu University,
   Higashi-ku, Fukuoka 812-8581, Japan%
}
\email{nakajo@math.kyushu-u.ac.jp}
\subjclass[2000]{
Primary 57R45; Secondary 53A10, 53B20.}

\begin{document}
\begin{abstract}
We construct a new representation formula 
for indefinite improper affine spheres 
in terms of two para-holomorphic functions 
and study singularities which appear in this representation formula.
As a result, it follows that cuspidal cross caps never
appear as the singularities on indefinite improper affine spheres
and so on. Comparison with other representation formulae are also
studied.
\end{abstract}
\maketitle
\section*{Introduction}
 Affine spheres are important objects in the study of affine differential geometry. 
They have close relations to the theory of minimal surfaces, 
Monge-Amp\`{e}re equations, Liouville equation, Tzitzeica equation and so on.
In this paper, we construct a representation formula for indefinite improper affine spheres in the affine 3-space 
(Theorem~\ref{thm:rep}), in terms of two para-holomorphic functions.
This may be regarded as an indefinite version of the representation formula
for locally strongly convex improper affine spheres obtained by A. Mart\'{i}nez \cite{Martinez}.
Here, indefinite improper affine spheres mean improper affine spheres with indefinite affine metric.
It is found that affine spheres which are represented by these formulae may have singularities 
and we investigate them. 
As a result, we have peculiar examples of indefinite improper affine spheres with singularities (Section~\ref{sec:ex}).  

 So far, various kinds of representation formulae for (improper) affine spheres are studied by several authors. Perhaps the earliest
version, which is discovered in early twentieth century, is due to Blaschke \cite{Blaschke}. In this representation formula, improper affine spheres are represented in terms of two smooth functions. Another representation formula
is due to Cort\'{e}s \cite{Cortes}. His formula represents special class of improper affine spheres which are related to special  K\"{a}hler structure, in terms of one holomorphic function. The feature of his representation formula is that it even covers higher (even-) dimensional improper affine spheres. 
 
 Recently, A. Mart\'{i}nez made a representation formula for locally strongly convex improper affine spheres 
in terms of two holomorphic functions (\cite{Martinez}). 
The feature of Mart\'{i}nez' representation formula is that it includes even improper affine spheres with singularities. 
In the same paper, he also introduced the notion of IA-map, a class of 
(locally strongly convex) improper affine spheres with singularities
which have close relations with his representation formula.
He also studied a correspondence between improper
affine spheres and flat fronts in hyperbolic 3-space (see \cite{KRSUY}, \cite{flat fronts}) in \cite{IGV}.

 Following Mart\'{i}nez' manner, we introduce \textit{generalized IA-maps} in 
Section~\ref{sec: rep}. (Definitions~\ref{def: Martinez}, ~\ref{def: indefinite}).
 The singularities on indefinite generalized IA-maps have different properties from those on locally strongly convex ones. 
The singularities on locally strongly convex generalized IA-maps are either singularity of fronts, or branch points.
Contrarily, the singularities on indefinite generalized IA-maps, however, are not necessarily fronts nor branch points. 
Moreover, the cuspidal cross caps never appear as singularities of indefinite generalized IA-maps (Theorem~\ref{thm: not 
CCR}), while they appear often in the case of spacelike maximal surfaces in Lorentzian 3-space \cite{FSUY}.
Furthermore, we found a strange example of singularity on indefinite generalized IA-map (Section~\ref{sec:ex}). 
This is not $\mathcal{A}$-equivalent to any singularities on improper affine spheres obtained by 
the projection of a generalized geometric solution  of certain Monge-Amp\`{e}re system 
which are considered in \cite{Ishikawa-Machida} (Section~\ref{sec:comparison}). 

\section{Preliminaries}
\label{sec: prel}
 Before explaining the main part of results, we would like to review some definitions and basic known facts about affine differential geometry, para-complex number and singularity theory.  

\subsection{Improper affine spheres}
 At first, we would like to introduce briefly the affine differential geometry. 
For the detailed exposition, see \cite{Li-Simon-Zhao} and \cite{Nomizu-Sasaki}.

Let $(f,\xi)$ be a pair of an immersion $f: M^n \rightarrow {\R}^{n+1}$ 
into the affine space ${\R}^{n+1}$ and the vector field $\xi$ on $M^n$ along $f$ 
which is transversal to $f_{*}(TM)$.
Then, the Gauss-Weingarten formula become
\begin{equation}
\begin{cases}
\label{eq: Gauss-Weingarten}
  D_{X}{f_*Y} &= f_*({\nabla}_XY)+g(X,Y)\xi, \\
  D_{X}\xi &= -f_*(SX) +\tau(X)\xi, \\
\end{cases}
\end{equation}
where $D$ is the standard connection on ${\R}^{n+1}$. Here, $g$ is called the affine metric of a pair $(f,\xi)$.
It can be easily shown that the rank of the affine metric $g$ is invariant 
under the change of transversal vector field $\xi$. So we call $f$ 
a \textit{locally strongly convex immersion} (respectively \textit{indefinite immersion})  
if $g$ is positive definite (resp. indefinite).
Given an immersion $f: M^n \rightarrow {\R}^{n+1}$, we can choose uniquely 
the transversal vector field $\xi$ which satisfy the following conditions,
\begin{enumerate}
\item ${\tau}\equiv 0$, (or equivalently $D_X\xi \in f_*(TM)$ \ for all $X \in \mathfrak{X}(M)$),
\item ${\vol}_g(X_1, \cdots, X_n) = {\det}(f_*X_1, \cdots, f_*X_n, \xi)$ \ for all  $X_1, \cdots, X_n \in \mathfrak{X}(M)$,
\end{enumerate}
where ${\vol}_{g}$ is the volume form of the (pseudo-)Riemannian metric $g$ and 
$\det$ is the standard volume element of ${\R}^{n+1}$. 
The transversal vector field $\xi$ which satisfies above two conditions 
is called a \textit{Blaschke normal} (or \textit{affine normal}) 
and a pair $(f,\xi)$ of a immersion and its Blaschke normal is called 
a \textit{Blaschke immersion}.

 A Blaschke immersion $(f,\xi)$ with $S=0$ in ~\eqref{eq: Gauss-Weingarten}
is called an \textit{improper affine sphere}.
In this case $\xi$ becomes constant vector because $\tau=0$.
Hence hereafter, we can think of a transversal vector field $\xi$ of 
a improper affine sphere as $\xi=(0,\cdots, 0,1)$
after certain affine transformation of ${\R}^{n+1}$.
 
 The \textit{conormal map} ${\nu}: M^n \rightarrow ({\R}^{n+1})^{*}$ for 
a given Blaschke immersion $(f, \xi)$ is defined as the immersion 
which satisfy the following conditions,
\begin{enumerate}
\item ${\nu}(f_{*}X) =0$ for all $X \in \mathfrak{X}(M)$,
\item ${\nu}(\xi)=1$.
\end{enumerate}
For an improper affine sphere with the Blaschke normal $(0, \cdots, 1)$,
we can write $\nu=(n,1)$ with a smooth map $n: M^2 \rightarrow {\R}^n$.
 
Using the notations defined as above, we can now state the representation formula 
for locally strongly convex improper affine spheres (possibly with singularities)
by Martinez in \cite{Martinez}.
In that paper, he define the notion of improper affine maps, 
a generalization of improper affine spheres which possibly have 
singularities (Definition~\ref{def: Martinez}) 
and give its representation formula as follows:  
\begin{theorem}[Theorem 3 in \cite{Martinez}]
\label{thm:lscIA-map}
 Let $\psi=(x,\phi): {\Sigma}^2\rightarrow {\R}^2\times {\R}$ be an improper affine map. Then
  there exists a regular complex curve ${\alpha}:=(F,G):{\Sigma}^2 \rightarrow {\C}^2$ such that,
\begin{equation}
\label{eq:IA-map}
   \psi=\left( G+\bar{F}, \dfrac{1}{2}(|G|^2-|F|^2)+{\Re}\left( GF-2\int{FdG} \right) \right).
\end{equation}
Here, the conormal map of $\psi$ becomes
\[
  \nu = (\bar{F}-G ,1).
\]
Conversely, given a Riemann surface ${\Sigma}$ and a complex curve $\alpha:=(F,G):{\Sigma}\rightarrow 
{\C}^2$, then ~\eqref{eq:IA-map} gives an improper affine map which is well defined if and only if
$\int{FdG}$ does not have real periods.
\end{theorem}

\subsection{Para-complex numbers}  
 The set of para-complex numbers $\widetilde{\C}$, is an algebra over $\R$ which is defined as
\[
 {\widetilde{\C}}:=\left\{a+jb| a,b \in {\R}, j^2=1 \right\}.
\]
For $z = u+jv \in {\widetilde{\C}}$ the conjugate $\bar{z}$ is defined as $\bar{z}:= u-jv$ 
and absolute value $|z|$ is defined as $|z|:=z{\bar{z}}=u^2-v^2$.
As an analogy to the holomorphic function, we can define the so-called para-holomorphic function: 
A map $F: {\widetilde{\C}}\rightarrow {\widetilde{\C}}$
is called a \textit{para-holomorphic function} if $F$ satisfies the para-Cauchy-Riemann equation, that is,
\begin{equation}
\label{eq: para-CR}
\begin{cases}
   &\dfrac{\partial{f^1}}{\partial{u}} =  \dfrac{\partial{f^2}}{\partial{v}}, \\[6pt]
   &\dfrac{\partial{f^1}}{\partial{v}} =  \dfrac{\partial{f^2}}{\partial{u}}. \\
\end{cases}
\end{equation}
where $F(u+jv)=f^1(u,v)+jf^2(u,v)$.
The set of para-holomorphic functions forms an algebra. For a para-holomorphic function
$F: {\widetilde{\C}}\rightarrow {\widetilde{\C}}$,  we define its differential $F'$ as
\[
F'(u+iv):=\dfrac{\partial{f^1}}{\partial{u}}+j\dfrac{\partial{f^2}}{\partial{u}} \ \ \ \ 
( \textrm{or equivalently}, F'(u+iv):=\dfrac{\partial{f^2}}{\partial{v}}+j\dfrac{\partial{f^1}}{\partial{v}})
\]
where $F(u+jv)=f^1(u,v)+jf^2(u,v)$.

Since ~\eqref{eq: para-CR} is reduced to the wave equation
\[
 \frac{{\partial}^2{f^j}}{\partial{u^2}}
 -\frac{{\partial}^2{f^j}}{\partial{v^2}}=0 \ \ \ \ \ (j=1,2)
\]
and the general solution of the wave equation, which is known 
as d'Alembert solution, is 
given by arbitrary two smooth functions,
a para-holomorphic function is  
described in terms of two smooth functions.
Concretely, for a para-holomorphic function $F$, there exist two smooth functions $\rho$ and $\sigma$ on ${\R}^2$,
such that 
\[ 
  F(u+jv) = \rho (u+v)+\sigma(u-v) + j \left\{ \rho(u+v)-\sigma(u-v) \right\}
\] 
holds. For more detailed exposition on para-complex numbers, see \cite{Inoguchi-Toda}.

\subsection{Criteria for singularities}
 Here, we explain criteria for singularities of smooth maps (see \cite{FSUY}, \cite{KRSUY} and \cite{flat fronts}).
The explanation is restricted to the case for smooth maps from $2$-dimensional manifolds $M^2$ to the affine
$3$-space ${\R}^3$ because all the affine spheres are considered to be $2$-dimensional
in this paper.

 Consider a smooth map $f:U\rightarrow {\R}^3$ from a open subset $U$ of ${\R}^2$
to the affine $3$-space ${\R}^3$. 
A point $p$ on $U$ is called a \textit{singular point of $f$}, if $f$ is not immersive
at $p \in U$. On studying the local properties of singularity on smooth maps, 
one usually consider its map-germ. Let $f:U\rightarrow {\R}^3$ and $g: V \rightarrow {\R}^3$
are smooth maps from open sets $U$, $V$ of ${\R}^2$ to ${\R}^3$ respectively, with $p \in U\cap V$.
Then we call \textit{$f$ and $g$ defines the same map-germ at $p$} if there exists an open set
$W \subset {\R}^2$  with $p \in W$ and $W \subset {U\cap V}$ such that $f=g$ holds on $W$.  
Next, we introduce $\mathcal{A}$-equivalence, an equivalence relation between two singularities.
Given two smooth maps $f:U\rightarrow {\R}^3$ and $g:\widetilde{U}\rightarrow {\R}^3$, and
assume $p \in U$ and $\tilde{p} \in {\widetilde{U}}$ are singular points. Consider two map-germ
$(f,p)$ around $p \in U$ and $(g, \widetilde{p})$ around $\tilde{p} \in \widetilde{U}$.
Then $(f,p)$ and $(g, \tilde{p})$ is called \textit{$\mathcal{A}$-equivalent} 
if there exists two diffeomorphism-germ $\phi: U\rightarrow {\widetilde{U}}$ and 
$\psi:{\R}^3\rightarrow {\R}^3$ such that $\psi \circ f = \widetilde{f} \circ \phi$
as a map-germ around $p \in U$ and $\tilde{p}= \phi(p)$ hold.

 In this paper, we mainly consider the special class of singularities, that is, frontals and fronts. 
A smooth map (possibly with singularities) $f:U\rightarrow {\R}^3$ from open subset $U \subset {\R}^2$
to the affine $3$-space ${\R}^3$ is called a \textit{frontal} (or a \textit{frontal map}) 
if there exists a unit normal vector field $\nu$ to $f$ (even on singular points). 
This is equivalent to the existence of 
a map to the 2-sphere $\nu:U \rightarrow S^2\subset{\R}^3$ which satisfies
\begin{equation}
\label{eq:frontal map}
\langle df, \nu \rangle = 0,
\end{equation}
where $\langle , \rangle$ is the Euclidean metric of ${\R}^3$.
By the identification of the unit tangent bundle $T_{1}{\R}^3$ and the unit cotangent bundle $T_{1}^{*}{\R}^3$,
frontals can be considered as the projection of a Legendrian map into $T_{1}^{*}{\R}^3$ with respect to 
the canonical contact structure as follows. Let $f:U\rightarrow {\R}^3$ be a frontal and $\nu$ be its unit normal
vector field. Then $f$ and $\nu$ defines a map into the unit tangent bundle, $(f,\nu):U \rightarrow T_1{\R}^3$.
We can consider this map as $(f,\nu):U \rightarrow T_1^{*}{\R}^3$ by the above identification.
Then the condition ~\eqref{eq:frontal map} is equivalent to that $(f,\nu)$ is Legendrian map into $T_{1}^{*}{\R}^3$.
Conversely, from a Legendrian map into $T_{1}^{*}{\R}^3$, we can make a frontal into ${\R}^3$ by the projections 
into ${\R}^3$ and $S^2$ respectively. A frontal $f:U\rightarrow {\R}^3$ is called a \textit{front} if $(f,\nu)$
defined as above is an immersion. From above arguments, it can be seen that 
the notion of fronts and frontals do not depend on the choice of 
the Riemannian metric $g$ of ${\R}^3$.

Let $\det$ be the standard volume element of ${\R}^3$. 
Then the function $\lambda(u,v):=\det(f_u, f_v, \nu)$ is called the \textit{signed area density function}
where $\tilde{nu}$ is the unit normal of $f$. It is obvious from the definition that $p \in U$ is a singular
point if and only if $\lambda(p)=0$ holds.
A singular point $p \in M^2$ is called \textit{non-degenerate}
if the derivative $d{\lambda}$ of the signed area density function does not vanish at $p$.

 The typical examples of fronts are \textit{cuspidal edges} and \textit{swallowtails}, while those of frontals are 
a \textit{cuspidal cross caps}. Here, cuspidal edges, swallowtails and cuspidal cross caps are defined as a smooth
map $f:M^2\rightarrow N^3$ which are $\mathcal{A}$-equivalent to the following maps, $f_C$, $f_S$ and $f_{CCR}$ 
respectively:
\[
  f_C(u,v):=(u^2,u^3,v), f_S(u,v):=(3u^4+u^2v,4u^3+2uv,v),
\]
\[   f_{CCR}(u,v):=(u,v^2,uv^3).
\]

 In \cite{FSUY} and \cite{KRSUY}, a geometric criteria for singular points to be cuspidal edges, swallowtails
and cuspidal cross caps are given. To make the criterion, they define \textit{singular curve},
\textit{null direction} and so on. If $p \in U$ is a non-degenerate singular point, then the implicit 
function theorem implies that the singular set $\Sigma_f$ around $p$ becomes locally a regular curve 
because $\Sigma_f$ coincides with the zero sets of signed area density function $\lambda$. This curve is called
a singular curve and denoted by $\gamma: (-\epsilon, \epsilon) \rightarrow U $. Usually, we fix the 
parametrization of $\gamma$ as to $\gamma(0)=p$. The \textit{singular direction} of $f$ at $\gamma(t)$ is the 
1-dimensional subspace of $T_pU$ which is spanned by $\gamma'(0) \in T_pU$. The null direction is 
defined as the kernel of the differential map $f_*$. This is defined uniquely only in the case that image of the
differential map $f_*$ become 1-dimensional. A vector field $\eta(t)$ along the singular curve
$\gamma(t)$ is called a \textit{null vector field} if it associates the non-zero vector belonging to 
the null direction in $T_{\gamma(t)}U$ to each $t$.

 Using the above definitions, criteria for cuspidal edges, swallow tails and cuspidal cross caps are given as 
follows.  
\begin{fact}[Criteria for cuspidal edges and swallowtails \cite{KRSUY}]
\label{fact: CS}
Let $p$ be a non-degenerate singular point of a front $f$, $\gamma$ the singular curve passing through $p$, 
and $\eta$ a null vector field along $\gamma$. Then
\begin{enumerate}
 \item $p=\gamma(0)$ is a cuspidal edge if and only if the null direction and the singular direction are 
       transversal, that is, $\det(\gamma'(0), \eta(0)) \not= 0$ holds, where $\det$ denotes the determinant
       of $2\times 2$ matrices and where we identify the tangent space in $T_pU$ with ${\R}^2$.
 \item $p=\gamma(0)$ is a swallowtail if and only if
       \[ 
          \det(\gamma'(0), \eta(0)) = 0 \ \ \ \ and\ \ \ \  \dfrac{d}{dt}{\Big|}_{t=0}\det(\gamma'(t), \eta(t)) \not= 0
       \]hold. 
\end{enumerate} 
\end{fact} 
 
\begin{fact}[Criterion for cuspidal cross caps \cite{FSUY}]
\label{fact: CCR}
 Let $f:U\rightarrow {\R}^3$ be a frontal and $\gamma(t)$ a singular curve on $U$ 
passing through a non-degenerate singular point $p=\gamma(0)$. We set
\[
  \Psi(t):=\det(\tilde{\gamma}', D_{\eta}^{f}\nu, \nu),
\]
where $\tilde{\gamma}=f\circ{\gamma}$, $D_{\eta}^{f}\nu$ is the canonical covariant 
derivative along a map $f$ induced from the standard connection on ${\R}^3$, and 
$'=d/dt$. Then the germ of $f$ at $p=\gamma(0)$ is $\mathcal{A}$-equivalent to
a cuspidal cross cap if and only if 
\begin{enumerate}
 \item $\eta(0)$ is transversal to $\gamma'(0)$,
 \item $\Psi(0)=$ and $\Psi'(0)\not=0$.
\end{enumerate}
\end{fact} 
  
\section{Representation formula for indefinite improper affine spheres}
\label{sec: rep}
 In this section, we introduce a representation formula for indefinite improper affine spheres by modifying Mart\'{i}nez'
 method in ~\cite{Martinez}. To do this, we should review a duality relation for improper affine spheres 
with indefinite affine metric. Throughout this paper, we fix its affine normal to $\xi=(0,0,1)$ by the appropriate 
affine transformation of ${\R}^3$. 

\begin{proposition}[Duality relations]
 \label{prop: duality}
\ Let $f:{\Sigma}^2\rightarrow {\R}^3$ be an indefinite improper affine sphere, $\xi$ its affine normal and ${\nu}:{\Sigma}^2\rightarrow {\R}^3$ the conormal 
map of $f$, where $\Sigma^2$ is an open subset of ${\R}^2$. Take a coordinate system $(u,v)$ of ${\Sigma}^2$ so that the affine metric $g$ of $f$ (which is indefinite in this 
case) is represented as $g=E(du^2-dv^2)$. Then, the following identities hold;
\[\hspace{0.1cm}
\left\{
\begin{array}{cll}
f_u&= {\nu}\times {\nu}_v \ \ \ {\nu}_u= f_v\times {\xi}\\
f_v&= {\nu}\times {\nu}_u \ \ \ {\nu}_v= f_u\times {\xi}\\
\end{array}
\right.
\]
where $\times$ is the outer product on ${\R}^3$, 
$f_u:={\dfrac{\partial{f}}{\partial{u}}}$ and $f_v:={\dfrac{\partial{f}}{\partial{v}}}$ 
\end{proposition}

\begin{proof}
 The duality relations for affine spheres are well-known for the other kinds of affine spheres. The proof is almost
parallel to them. 
In fact, we can obtain the duality relations for indefinite improper affine spheres by taking care of the definition of
conormal maps and the 
Gauss-Codazzi equations. The case of definite improper affine spheres is presented in \cite{Matsuura-Urakawa}.
\end{proof}

 Let $\psi=(x,\phi):{\Sigma}^2\rightarrow {\R}^3(={\R}^2\times{\R})$ be 
an indefinite improper affine sphere with the Blaschke normal ${\xi}=(0,0,1)$, 
where $x:{\Sigma}^2\rightarrow {\R}^2$ and $\phi:{\Sigma}^2\rightarrow {\R}$ are smooth maps. 
 Then the conormal vector field ${\nu}:{\Sigma}^2\rightarrow {\R}^3(={\R}^2\times{\R})$ can be represented 
as ${\nu}=(n,1)$ and the affine metric is written as $g=-\langle dx, dn \rangle $, where $\langle\ ,\ \rangle$ 
is the standard Euclidean inner product of ${\R}^2$. 

 Using Proposition~\ref{prop: duality}, we obtain the following characterization of indefinite improper affine spheres. 
 
\begin{proposition}
 \label{prop:duality}
 Define ${\tilde{L}}_{\psi}:{\Sigma}^2 \rightarrow {\widetilde{\C}}^2$ as 
\[{\tilde{L}}_f:=x+jn, 
\]
where $x$ and $n$ are defined as above, that is, $x:=\pi \circ {\psi}$ and $n:=\pi \circ \nu$ 
with the projection $\pi:{\R}^3 \rightarrow {\R}^2$.
Identify ${\widetilde{\C}}^2$ with ${\R}^4$ as,
\[{\widetilde{\C}}^2 \ni z_1=y_0+jy_1, z_2=y_2+jy_3 \longleftrightarrow   (y_0,y_1,y_2,y_3) \in {\R}^4
\]
Then $\tilde{L}_{\psi}$ annihilates $2$-forms $dy_0{\wedge}dy_1+dy_2{\wedge}dy_3$ and $dy_0{\wedge}dy_2+dy_1{\wedge}dy_3$, in other words, 
\begin{equation}
 \label{eq:2-forms}
 \begin{cases}
    dy_0{\wedge}dy_1+dy_2{\wedge}dy_3|_{{\tilde{L}}_{\psi}({\Sigma}^2)} &= 0, \\  
    dy_0{\wedge}dy_2+dy_1{\wedge}dy_3|_{{\tilde{L}}_{\psi}({\Sigma}^2)} &= 0. \\
  \end{cases} 
\end{equation}
hold.
\end{proposition}
 
\begin{proof}
 Applying Proposition~\ref{prop:duality}, we have
\[\begin{split}
   {{\tilde{L}}_{\psi}}^{*}({dy_0{\wedge}dy_1+dy_2{\wedge}dy_3}) &=dx^1{\wedge}dx^2+dn^1{\wedge}dn^2 \\
                   &=\left\{(x^1_ux^2_v-x^1_vx^2_u)+(n^1_un^2_v-n^1_vn^2_u)\right\}du{\wedge}dv \\  
                                                 &= 0 
  \end{split} 
\]
and 
\[\begin{split}
   {{\tilde{L}}_{\psi}}^{*}({dy_0{\wedge}dy_2+dy_1{\wedge}dy_3}) &=dx^1{\wedge}dn^1+dx^2{\wedge}dn^2 \\
                   &=\left\{(x^1_un^1_v-x^1_vn^1_u)+(x^2_un^2_v-x^2_vn^2_u)\right\}du{\wedge}dv \\  
                                                 &= 0 
  \end{split} 
\]
where $x=(x^1,x^2)$ and $n=(n^1,n^2)$.
\end{proof}

 Proposition~\ref{prop:duality} immediately leads us to the following representation formula.

\begin{theorem}
 \label{thm:rep}
 \begin{enumerate}
  \item Let ${\psi}=(x,\phi):{\Sigma}^2 \rightarrow {\R}^2 \times {\R}$ be an
	indefinite improper affine sphere whose affine normal is fixed
	to ${\xi}=(0,0,1)$. 
	Let ${\nu}=(n,1):{\Sigma}^2 \rightarrow {\R}^2 \times {\R}$ be
	the conormal map of $\psi$. Then there exists a para-holomorphic curve
	${\alpha}:=(F,G):{\Sigma}^2 \rightarrow {\widetilde{\C}}^2$ such
	that $|dF|\not=|dG|$ and
	\begin{equation}
	 \label{eq:F and G}
  \begin{cases}
   x&= {F}-{\bar{G}}, \\
   n&= {\bar{F}}+{G}. \\
  \end{cases} 
	\end{equation}
	hold everywhere.
 \item Conversely, given a para-holomorphic curve $(F,G):
	{\Sigma}\rightarrow {\widetilde{\C}}^2$, then 
	\[ \psi=\left(x, -\int{{\langle\ n,dx  \rangle}}\right)
	\]
	gives an indefinite improper affine sphere (possibly with
	singularities), where $x$ and $n$ are defined by ~\eqref{eq:F and G}. 
	Moreover, if we write $F$ and $G$ as $F=f^1+jf^2, G=g^1+jg^2$,
	then $\psi$ can be expressed as
	\begin{multline}
    \label{eq:explicit}
	 \psi=\left(
	 f^1-g^1,
	 f^2+g^2,
         \vphantom{\int}\right.\\
	 -\int{\left\{(f^1+g^1)(f^1_u-g^1_u)+(-f^2+g^2)(f^2_u+g^2_u)\right\}du-}
            \\
	 \left.
	 \int{\left\{(f^1+g^1)(f^2_u-g^2_u)+(-f^2+g^2)(f^1_u+g^1_u)\right\}dv}
	 \right)
	\end{multline}
 \end{enumerate}
\end{theorem}

\begin{proof}
\begin{enumerate}
\item
  First, define $F:{\Sigma} \rightarrow {\widetilde{\C}}$ as
\[ F:=f^1+jf^2:={\frac{x^1+n^1}{2}}+j{\frac{x^2-n^2}{2}}. \]
Then, by Proposition~\ref{prop:duality}, we have
\[ f^1_u=f^2_v, \ \ \ f^1_v=f^2_u \]
and $F$ is a para-holomorphic function on ${\Sigma}^2$ with respect to the para-complex structure induced from $g$. \\
\ Take a new coordinate system $\left\{w_1, w_2 \right\}$ on ${\widetilde{\C}}^2$ as,
\[\begin{cases}
    w_1 &={\dfrac{y_0+y_1}{2}}+j{\dfrac{y_2-y_3}{2}}, \\
    w_2 &={\dfrac{-y_0+y_1}{2}}+j{\dfrac{y_2+y_3}{2}}. \\
 \end{cases} \]
Then,  \eqref{eq:2-forms} is equivalent to
\begin{equation}
 \label{eq:2-form}
   dw_1{\wedge}dw_2|_{{\tilde{L}}_{\psi}({\Sigma})}=0.
\end{equation}
Here $w_1=F$ and $w_1$ induce the same para-complex structure on ${\Sigma}$ as the one induced from $g$. Thus \eqref{eq:2-form} implies that $w_2$ also defines a 
para-holomorphic function on ${\Sigma}$. Define a para-holomorphic function $G:{\Sigma} \rightarrow {\tilde{\C}}$ as $G:=w_2$, then we have \eqref{eq:F and G}. The condition that $|dF|\not=|dG|$
is equivalent to the absence of singularity on $\psi$.
\item We can check that $\psi$ defines an indefinite improper affine sphere by direct calculation. However, it is almost obvious that $\psi$ 
defines an indefinite improper affine sphere by the following reason. For $\psi:{\Sigma}^2\rightarrow {\R}^3$, a duality relation is equivalent to the condition to be an affine sphere. On the other hand, a duality relation is equivalent to para-Cauchy-Riemann equation
in this case. 
\end{enumerate}
\end{proof}

Next, we introduce the notion of IA-maps and generalized IA-maps, which is a generalization of improper affine sphere. IA-maps are firstly defined by A. Mart\'{i}nez in \cite{Martinez}, as a generalization of improper affine sphere which permits a singularity. Modifying the Mart\'{i}nez' definition of IA-maps, we define generalized IA-maps as the following.

\begin{definition}[Definition 1 in \cite{Martinez}]
\label{def: Martinez}
 A map $\psi=(x,\phi):{\Sigma}^2\rightarrow {\R}^3(={\R}^2\times{\R})$, where $x:{\Sigma}^2\rightarrow {\R}^2$ 
and $\phi:{\Sigma}^2\rightarrow {\R}$ are smooth maps, is called an \textit{IA-map} (respectively \textit{a generalized IA-map}) if ${\tilde{L}}_{\psi}:{\Sigma}^2 \rightarrow {\C}^2$, which is defined by 
\[{\tilde{L}}_{\psi}:=x+\sqrt{-1}n, 
\]
becomes a SL-immersion (resp. SL-map) with respect to the symplectic structure ${\omega}'$ and the calibration $\Re(\sqrt{-1}{\Omega}')$, 
where ${\omega}'=\dfrac{\sqrt{-1}}{2}(d{\zeta}_1 {\wedge} d{\bar{{\zeta}_1}}+d{\zeta}_2 {\wedge} d{\bar{{\zeta}_2}})$ 
and ${\Omega}'=d{\zeta}_1 {\wedge} d{\zeta}_2$.
\end{definition}

 As an analogy to this definition, it is appropriate to define the class of indefinite improper affine spheres with singularities as
follows.

\begin{definition}
 \label{def: indefinite}
 A map $\psi=(x,\phi):{\Sigma}^2\rightarrow {\R}^3(={\R}^2\times{\R})$, where $x:{\Sigma}^2\rightarrow {\R}^2$ 
and $\phi:{\Sigma}^2\rightarrow {\R}$ are smooth maps, is called an \textit{indefinite IA-map} (respectively, an \textit{indefinite generalized IA-map}) if ${\tilde{L}}_{\psi}:{\Sigma}^2 \rightarrow {\widetilde{\C}}^2$, which is defined by 
\[{\tilde{L}}_{\psi}:=x+jn, 
\]
becomes an immersion (resp. a map which is not necessarily an immersion) and 
${\tilde{L}}_{\psi}^{*}(dy_0{\wedge}dy_1+dy_2{\wedge}dy_3)
 ={\tilde{L}}_{\psi}^{*}(dy_0{\wedge}dy_2+dy_1{\wedge}dy_3)=0$
holds.
\end{definition}

From now on, to avoid the confusion, we call an IA-map (respectively, a generalized IA-map) with positive definite affine
metric an \textit{locally strongly convex IA-maps} (resp. an \textit{locally strongly convex generalized IA-map}).

\section{Singularities of IA-maps}
The representation formula for locally strongly convex improper affine spheres and that for indefinite ones look very similar. 
However, singular points which appear in two representation formulae have quite different properties each other.
Concretely, in the locally strongly convex case, except for on branch points, all the singular points are fronts. 
On the other hand, in the indefinite case, a singular point is 
not necessarily a front even if it is not a branch point. Indeed, we can easily find examples of frontal maps which are not fronts in indefinite case by direct calculations. 

\subsection{Locally strongly convex case}
 To study the properties of singularities on a locally strongly convex IA-map ${\psi}:{\Sigma}^2\rightarrow {\R}^3$,
we will first check the condition for $p \in \Sigma^2$ to be a singular point of $\psi$, which is also given in 
\cite{Martinez}.

\begin{proposition}[\cite{Martinez}]
 \label{prop: sing. pt. on l.s.c.}
 Let $\psi:{\Sigma} \rightarrow {\R}^{3}$ be an locally strongly convex generalized IA-map¡¥
From the representation formula ~\eqref{eq:IA-map} for locally strongly convex IA-maps in
Theorem~\ref{thm:lscIA-map},
$p\in{\Sigma}^2$ is a singular point of $\psi$ if and only if $|dF|=|dG|$ holds on $p\in{\Sigma}^2$. 
\end{proposition}
\begin{proof}
 From the explicit form of the representation formula ~\eqref{eq:explicit}, the differentiation of $\psi$ becomes
\[\begin{split}
  {\psi}_u &= (f^1_u+g^1_u, g^2_u-f^2_u, (g^1-f^1)(g^1_u+f^1_u)+(g^2+f^2)(g^2_u-f^2_u)), \\
  {\psi}_v &= (-f^2_u-g^2_u, g^1_u-f^1_u, (-g^1+f^1)(g^2_u+f^2_u)+(g^2+f^2)(g^1_u-f^1_u)). \\
\end{split}\]
This can be rewritten as 
\begin{equation}
 \label{eq: indefinite}
 \begin{split}
  {\psi}_u &= (f^1_u+g^1_u)(1,0,g^1-f^1)+(g^2_u-f^2_u)(0,1,g^2+f^2) \\
  {\psi}_v &= -(g^2_u+f^2_u)(1,0,g^1-f^1)+(g^1_u-f^1_u)(0,1,g^2+f^2). \\
\end{split}
\end{equation}
The linearly independence of the above two vectors implies that the condition to be a singular points is
\[
 (f^1_u+g^1_u)(g^1_u-f^1_u)-(g^2_u-f^2_u)\left\{-(g^2_u+f^2_u)\right\}=0 .
\]
This is equivalent to $|dF|=|dG|$. 
\end{proof}

Next we show that the singularity of locally strongly convex improper affine spheres 
is always locally a front unless it is a branch point, that is, $\psi_u=\psi_v=0$ holds 
on that point. 
In other words, Legendrian lift $L_{\psi}=(\psi,{\tilde{\nu}}): {\Sigma}^2
\rightarrow {T_1{\R}^3}$ is immersive on ${\Sigma}^2$ (even on the singular point of $\psi$), 
where $\tilde{\nu}$ is the unit normal of $\psi$.

\begin{proposition}
 \label{prop:front}
 Singularity of a locally strongly convex generalized IA-map is always locally a front if $|dF|=|dG|=0$ does not hold. 
\end{proposition}

\begin{proof}
 The proof follows from the direct calculation. We check that ${\Ker}d{\psi} \not= {\Ker}{d{\tilde{\nu}}}$ holds 
at singular points of $\psi$.

From Theorem~\ref{thm:lscIA-map}, the explicit representation of ${\tilde{\nu}}$ is
\[
   \tilde{\nu} = \frac{1}{(1+(f^1-g^1)^2+(f^2+g^2)^2)^{\frac{1}{2}}}(f^1-g^1, -(f^2+g^2), 1),
\]and its differentiations are
\begin{alignat*}{2}
 {\tilde{\nu}}_u &=
         \frac{1}{{\sqrt{1+(f^1-g^1)^2+(f^2+g^2)^2}}^3}\cdot\\
        &
        \hspace{1in}
        \begin{aligned}
        \biggl(
           &(f^1_u-g^1_u)(1+(f^2+g^2)^2)-(f^2+g^2)(f^1-g^1)(f^2_u+g^2_u), \\
           &-(f^2_u+g^2_u)(1+(f^1-g^1)^2)+(f^2+g^2)(f^1-g^1)(f^1_u-g^1_u), \\
           &-(f^1-g^1)(f^1_u-g^1_u)-(f^2+g^2)(f^2_u+g^2_u)
        \biggr),
        \end{aligned} \\
 {\tilde{\nu}}_v &=
         \frac{1}{{\sqrt{1+(f^1-g^1)^2+(f^2+g^2)^2}}^3}\cdot\\
        &
        \hspace{1in}
        \begin{aligned}
        \biggl(
           &-(f^2_u-g^2_u)(1+(f^2+g^2)^2)-(f^2+g^2)(f^1-g^1)(f^1_u+g^1_u), \\
           &-(f^1_u+g^1_u)(1+(f^1-g^1)^2)-(f^2+g^2)(f^1-g^1)(f^2_u-g^2_u), \\
           &(f^1-g^1)(f^2_u-g^2_u)-(f^2+g^2)(f^1_u+g^1_u)
        \biggr).
        \end{aligned}        
\end{alignat*}
 The proof of Proposition~\ref{prop: sing. pt. on l.s.c.} implies that
\begin{equation}
(g^2_u+f^2_u){\psi}_u+(g^1_u+f^1_u){\psi}_v = 0
\end{equation}
holds at a singular point of $f$. Then, to prove ${L_{\psi}}$ is an immersion, we have only to show that
\begin{equation}
(g^2_u+f^2_u){\tilde{\nu}}_u+(g^1_u+f^1_u){\tilde{\nu}}_v \not= 0 
\end{equation}
holds at any singular point.\\
\ Here,
\[
 \begin{split}
    &{{\sqrt{1+(f^1-g^1)^2+(f^2-g^2)^2}}^3}\left\{(g^2_u+f^2_u){\tilde{nu}}_u+(g^1_u+f^1_u){\tilde{\nu}}_v\right\} \\
   =(&-(g^1-f^1)(g^2+f^2)((g^1_u+f^1_u)^2+(g^2_u+f^2_u)^2)+2(1+(f^2+g^2)^2)(f^1_ug^2_u-f^2_ug^1_u), \\
     & -(1+(f^1-g^1)^2)((g^1_u+f^1_u)^2+(g^2_u+f^2_u)^2)+2(f^1-g^1)(f^2+g^2)(f^1_ug^2_u-f^2_ug^1_u), \\
     & -(f^2+g^2)((f^1_u+g^1_u)^2+(g^2_u+f^2_u)^2)-2(f^1-g^1)(f^1_ug^2_u-f^2_ug^1_u) ) \\
   =(&-(g^1-f^1)(g^2+f^2)A+2(1+(f^2+g^2)^2)B, -(1+(f^1-g^1)^2)A+2(f^1-g^1)(f^2+g^2)B, \\
     & -(f^2+g^2)A-2(f^1-g^1)B ). \\
 \end{split}
\]
where, $A=(f^1_u+g^1_u)^2+(g^2_u+f^2_u)^2$ and $B=(f^1_ug^2_u-f^2_ug^1_u)$.
Here $A>0$ because $|dF|=|dG|=0$ does not hold.
 If\[
{{\sqrt{1+(f^1-g^1)^2+(f^2-g^2)^2}}^3}\left\{(g^2_u+f^2_u){\tilde{\nu}}_u+(g^1_u+f^1_u){\tilde{\nu}}_v\right\}=0
\]
holds, then the computation of $B/A$ by using the fact that each element becomes 0 gives $1+(f^1-g^1)^2+(f^2+g^2)^2=0$. However, this is contradiction.
So
\[
(g^2_u+f^2_u){\nu}_u+(g^1_u+f^1_u){\nu}_v \not= 0 
\]
holds at any point (even at a singular point)¡¥\\
\ This completes the proof of Proposition~\ref{prop:front}. 
\end{proof}

\begin{remark}
\label{rmk: branch}
 We can state the relation between $\psi$, $\tilde{L}_{\psi}$ and $L_{\psi}$ as follows. 
$L_{\psi}$ is an immersion if and only if ${\tilde{L}}_{\psi}$ is an immersion.
So a locally strongly convex generalized 
IA-map $\psi$ is a front if and only if ${\tilde{L}}_{\psi}$ is an immersion. 

On the other hand, if ${\tilde{L}}_{\psi}$ is not immersive, then ${\psi}_u={\psi}_v=0$ holds at that point. 
On such a point $|dF|=|dG|=0$ holds. We call such a point a branch point.
Note that branch points are isolated because $F$ and $G$ are both holomorphic functions. 
Above fact can be considered as another proof of Proposition\ref{prop:front}.
\end{remark}

\subsection{Indefinite case}
 So far in this section, we consider the singularity of locally strongly convex IA-maps. 
Unlike the locally strongly convex case, any indefinite improper affine sphere (with singularities) 
is not necessarily a front. Indeed, we can construct several concrete examples 
which are frontal maps, but not fronts. So we shall check the property of 
singularity of indefinite improper affine spheres in the rest of this section.

First, similar to the locally strongly convex case, we check the condition for a pair of para-holomorphic functions so that 
the corresponding indefinite improper affine 
sphere has a singular point. 
\begin{proposition}
 \label{prop: sing. pt. on indefinite}
 Let $\psi:{\Sigma}^2 \rightarrow {\R}^{3}$ be an indefinite generalized IA-map¡¥ 
Then, $p\in{\Sigma}^2$ is a singular point of $\psi$ if and only if $|dF|=|dG|$ holds on $p\in{\Sigma}^2$. 
\end{proposition}

\begin{proof}
 The proof is almost the same as that of Theorem\ref{prop: sing. pt. on l.s.c.} except for the detailed calculation.
 
 From the explicit form of the representation formula ~\eqref{eq:explicit}, the differentiation of $\psi$ becomes
\[\begin{split}
  {\psi}_u &= (f^1_u-g^1_u, f^2_u+g^2_u, -(f^1+g^1)(f^1_u-g^1_u)-(-f^2+g^2)(f^2_u+g^2_u)), \\
  {\psi}_v &= (f^2_u-g^2_u, f^1_u+g^1_u, -(f^1+g^1)(f^2_u-g^2_u)-(-f^2+g^2)(f^1_u+g^1_u)). \\
\end{split}\]
This can be rewritten as 
\begin{equation}
\label{eq:frontal}
\begin{split}
  {\psi}_u &= (f^1_u-g^1_u)(1,0,-(f^1+g^1))+(f^2_u+g^2_u)(0,1,f^2-g^2) \\
  {\psi}_v &= (f^2_u-g^2_u)(1,0,-(f^1+g^1))+(f^1_u+g^1_u)(0,1,f^2-g^2). \\
\end{split}
\end{equation}
The linearly independence of the above two vectors implies that the condition to be a singular points is
\[
 (f^1_u-g^1_u)(f^1_u+g^1_u)-(f^2_u+g^2_u)(f^2_u-g^2_u)=0 .
\]
This is equivalent to $|dF|=|dG|$. 
\end{proof}

As mentioned above, there exist frontal maps which are not fronts in the representation formula ~\eqref{eq:explicit}. 
Next, we shall check the condition for an indefinite
improper affine sphere to have this type of singularity.

\begin{proposition}
 \label{prop:not front}
 For an indefinite generalized IA-map $\psi:{\Sigma}^2 \rightarrow {\R}^{3}$, the following two conditions are equivalent:
\begin{enumerate}
 \item $\psi$ is a frontal map which is not a front at $p \in {\Sigma}^2$ and $p$ is not a branch point.
 \item One of the following two conditions are satisfied at $p\in{\Sigma}^2$.
  \begin{enumerate}
    \item $f^1_u=f^2_u$ \ and \ $g^1_u=g^2_u$,
    \item $f^1_u=-f^2_u$ \ and \ $g^1_u=-g^2_u$.
  \end{enumerate}
\end{enumerate}
\end{proposition}

\begin{proof}
 From ~\eqref{eq:frontal}, we already know that at any point, the tangent space of 
any indefinite generalized IA-map is spanned by two vectors $(1,0,-(f^1+g^1))$ and $(0,1,f^2-g^2)$. 
This means every infinite generalized IA-map has an unit normal, which is parallel to $(f^1+g^1, -f^2+g^2, 1)$. 
So every indefinite generalized IA-map is a frontal map. Therefore the condition (1)
is satisfied if and only if $\psi$ is not a front, i.e., $L_{\psi}$ is not a immersion. 
Thus we have only to check the condition for $L_{\psi}$ not to be an immersion. 

From Theorem\ref{thm:rep}, the unit normal ${\tilde{\nu}}$ of $\psi$ is 
\[
   {\tilde{\nu}} = \frac{1}{\sqrt{\Delta}}(f^1+g^1, -f^2+g^2, 1)
\]where $\Delta=(f^1)^2+(f^2)^2+(g^1)^2+(g^2)^2+2f^1g^1-2f^2g^2+1$. \\
\ Then, we have 
\begin{equation}
\label{eq:unit normal}
\begin{split}
   {\tilde{\nu}}_u &= -{\frac{{\Delta}_u}{2{\sqrt{\Delta}}^3}}(f^1+g^2, -f^2+g^2, 1) + {\frac{1}{\sqrt{\Delta}}}(f^1_u+g^1_u, -f^2_u+g^2_u, 0) \\
                   &= \frac{1}{2{\sqrt{\Delta}}^3} \left\{-\Delta_u(f^1+g^1, -f^2+g^2,1) + 2\Delta(f^1_u+g^1_u, -f^2_u+g^2_u, 0) \right\} \\
                   &= \frac{1}{2{\sqrt{\Delta}}^3} (-2(f^2_u-g^2_u)(f^2-g^2)(f^1+g^1)+2(-f^1_u+g^1_u)(f^2-g^2)^2+2(f^1_u+g^1_u), \\
                   & \ \ \ \ \ \ \ \ \ \ \ \ \ \ 2(f^1_u+g^1_u)(f^1+g^1)(f^2-g^2)+2(-f^2_u+g^2_u)(f^1+g^1)^2+2(-f^2_u+g^2_u), \\
                   & \ \ \ \ \ \ \ \ \ \ \ \ \  -2(f^1_u+g^1_u)(f^1+g^1)-2(f^2_u-g^2_u)(f^2-g^2) ), \\
   {\tilde{\nu}}_v &= -{\frac{\Delta_v}{2{\sqrt{\Delta}}^3}}(f^1+g^2, -f^2+g^2, 1) + {\frac{1}{\sqrt{\Delta}}}(f^2_u+g^2_u, -f^1_u+g^1_u, 0)  \\
                   &= \frac{1}{2{\sqrt{\Delta}}^3} \left\{ -\Delta_v(f^1+g^1, -f^2+g^2,1) + 2\Delta(f^2_u+g^2_u, -f^1_u+g^1_u, 0) \right\} \\
                   &= \frac{1}{2{\sqrt{\Delta}}^3} (-2(f^1_u-g^1_u)(f^1+g^1)(f^2-g^2)+2(f^2_u+g^2_u)(f^2-g^2)^2+2(f^2_u+g^2_u), \\
                   & \ \ \ \ \ \ \ \ \ \ \ \ \ \ 2(f^2_u+g^2_u)(f^1+g^1)(f^2-g^2)+2(-f^1_u+g^1_u)(f^1+g^1)^2+2(-f^1_u+g^1_u), \\
                   & \ \ \ \ \ \ \ \ \ \ \ \ \  -2(f^2_u+g^2_u)(f^1+g^1)-2(f^1_u-g^1_u)(f^2-g^2) ). \\  
\end{split}
\end{equation}
First, we consider the case that neither $f^1_u+g^1_u$, $f^1_u-g^1_u$, $f^2_u+g^2_u$ nor $f^2_u-g^2_u$
does not equal to $0$.
Since $|dF|=|dG|$ holds on a singular point of $\psi$, ~\eqref{eq:frontal} implies that
\begin{equation}
\label{eq:f}
   {\psi}_u=\frac{f^1_u-g^1_u}{f^2_u-g^2_u}{\psi}_v= \frac{f^2_u+g^2_u}{f^1_u+g^1_u}{\psi}_v
\end{equation}also holds on a singular point of $\psi$.
On the other hand, by the same reason, ~\eqref{eq:unit normal} implies that
\begin{equation}
\label{eq:nu}
{\tilde{\nu}}_u = \frac{f^2_u-g^2_u}{f^1_u-g^1_u}{\tilde{\nu}}_v= \frac{f^1_u+g^1_u}{f^2_u+g^2_u}{\tilde{\nu}}_v
\end{equation}
holds on a singular point of $\psi$.

By comparing ~\eqref{eq:f} and ~\eqref{eq:nu}, and taking care of the fact that $\psi$ is 
not a front if and only if $(L_{\psi})_u$ and $(L_{\psi})_v$ are parallel, we have 
condition (2) after elementary calculation.

In the case that either $f^1_u+g^1_u$, $f^1_u-g^1_u$, $f^2_u+g^2_u$ or $f^2_u-g^2_u$ equals $0$, we have
condition (2) by taking care of the fact that $\psi$ is not a front at $p$ and that $p$ is not
a branch point.
\end{proof}

\begin{remark}
In the same manner as Remark~\ref{rmk: branch}, we have the following relation between $\psi$, 
$\tilde{L_{\psi}}$ and $L_{\psi}$ for an indefinite generalized IA-map.

$L_{\psi}$ is an immersion if and only if ${\tilde{L}}_{\psi}$ is an immersion.
So an indefinite generalized IA-map $\psi$ is a front if and only if 
${\tilde{L}}_{\psi}$ is an immersion. 

On the other hand, if $\tilde{L_{\psi}}$ is not an immersion, then $|dF|=|dG|=0$ hold and
it follows that $\psi$ is a frontal which is not a front or $p$ is a branch point.   
\end{remark}

Now, we end this section by remarking that cuspidal cross caps never appear as 
the singularities of indefinite generalized IA-maps. This fact is proved by 
using the criterion for cuspidal cross caps (Fact~\ref{fact: CCR}) and the condition 
for indefinite generalized IA-map to be a frontal but not a front (Proposition~\ref{prop:not front}).

\begin{theorem}
\label{thm: not CCR}
 Let $\psi:{\Sigma}^2\rightarrow {\R}^3$ be an indefinite generalized IA-map 
and $p\in{\Sigma}^2$ be a point. Then the germ of $\psi$ at $p$ is not $\mathcal{A}$-equivalent 
to the cuspidal cross cap. 
\end{theorem}

\begin{proof}
 From Fact~\ref{fact: CCR}, it is sufficient to check that  ${\Psi}'(0) = 0$. 
Here, we have only to consider the case of $f^1_u=f^2_u \ and\  g^1_u=g^2_u$ 
because the proof of another case is parallel to it.
Since ${\gamma}'(0)=(1,1)$, $\eta(0)=(1,-1)$ and ${\nu}_{uu}(p)={\nu}_{vv}(p)$, we have ${\Psi}'(0) = 0$. 
\end{proof}

\begin{remark}
\label{rmk: IM}
 Theorem\ref{thm: not CCR} is proved in \cite{Ishikawa-Machida} for another setting. In \cite{Ishikawa-Machida}, Ishikawa and Machida considered singularities on improper affine spheres given by projection derived from the framework of differential systems. They also have
the same results as Theorem~\ref{thm: not CCR}.
In Section~\ref{sec:comparison}, we will discuss the relationship 
between the singularities on generalized indefinite IA-maps and
those which appear in \cite{Ishikawa-Machida}.
\end{remark}

\section{Comparison with other representation formulae}
 In addition to Mart\'{i}nez type representation formulae, there are several representation formulae 
for improper affine spheres. In this section, we compare the representation formulae for 
(mainly indefinite) improper affine spheres which are obtained in the previous sections, 
to the other representation formulae for improper affine spheres which are studied in \cite{Blaschke}, 
\cite{Cortes}, \cite{Cortes-Lawn-Schafer} and \cite{Ishikawa-Machida}. The idea that relates the Mart\'{i}nez' 
representation formula to the Cort\'{e}s' representation formula is originally due to T. Kurose 
(\cite{Kurose lecture}).

  At first, the relation between Mart\'{i}nez type representation formula and 
that obtained by Cort\'{e}s-Lawn-Sch\"{a}fer is as follows.
\begin{remark}
  For 2-dimensional case, we can derive Cort\'{e}s-Lawn-Sch\"{a}fer representation formula in \cite{Cortes-Lawn-Schafer} from Theorem~\ref{thm:rep} by taking $ F=\frac{1}{2}(z-jf'(z)), G=\frac{1}{2}(z+jf'(z)) $. 
\end{remark}
 
 This fact is the analogy for the case of locally strongly convex ones as below (Remark~\ref{rmk: lsc}).
 
 The following two remarks are pointed out by T. Kurose.
  
\begin{remark}[\cite{Kurose lecture}]
 \label{rmk: lsc}
 For 2-dimensional case, we can derive Cort\'{e}s' representation formula in \cite{Cortes} from Mart\'{i}nez' representation formula in \cite{Martinez} by taking $ F=\frac{1}{2}(z-\sqrt{-1}f'(z)), G=\frac{1}{2}(z+\sqrt{-1}f'(z)) $. 
\end{remark}  

\begin{remark}
 By the identification of a para-holomorphic function to a pair of smooth functions explained in Section~\ref{sec: prel},
we can derive Blaschke's representation formula from the Mart\'{i}nez type representation formula in Theorem~\ref{thm:rep}. 

 Here, two para-holomorphic functions $F,G$ can be written as 
\[
\begin{cases}
 F(u,v)&={\rho}_1(u+v)+{\sigma}_1(u-v)+j\left\{{\rho}_1(u+v)-{\sigma}_1(u-v)\right\}, \\
 G(u,v)&={\rho}_2(u+v)+{\sigma}_2(u-v)+j\left\{{\rho}_2(u+v)-{\sigma}_2(u-v)\right\} \\
\end{cases} 
\]
by four smooth function ${\rho}_1, {\rho}_2, {\sigma}_1, {\sigma}_2$. Therefore if we take 
\[
\begin{cases}
 U_1&= {\rho}_1+{\rho}_2, \ \ V_1 = {\sigma}_1+{\sigma}_2, \\
 U_2&= -{\rho}_1+{\rho}_2, \ \ V_2={\sigma}_1-{\sigma}_2, \\
\end{cases} 
\]
then we have the Blaschke's representation formula in \cite{Blaschke}.

\end{remark}

\section{Examples}
\label{sec:ex}
 In \cite{Martinez}, Mart\'{i}nez gave several examples of locally strongly convex 
improper affine spheres by substituting holomorphic functions to his representation formula. 
Here, we construct indefinite improper affine spheres in the same manner as \cite{Martinez}.

\begin{example}
\label{ex: 23}
  Let $(F,G)=( z^2, z^3)$, then Theorem~\ref{thm:rep} gives a concrete example of indefinite generalized IA-maps, 
which is a frontal map but not a front.

For this example, we have
 \[
\begin{cases}
  f^1&=u^2+v^2, \ \ f^2=2uv, \\
  g^1&=u^3+3uv^2, \ \ g^2=3u^2v+v^3, \\ 
\end{cases} 
\]and this implies
\[
 {\lambda}=(f^1_u)^2-(f^2_u)^2-(g^1_u)^2+(g^2_u)^2 =(u^2-v^2)(4-9(u^2-v^2))
\]and
\[
\begin{cases}
 {\lambda}_u&=4u(2-9u(u^2-v^2)), \\
 {\lambda}_v&=-4v(2-9v(u^2-v^2)). \\
\end{cases} 
\]
Therefore, singular locus is $\left\{(u,v)\in{\R}^2|(u^2-v^2)(4-9(u^2-v^2))=0\right\}$. Among them, only the origin is degenerate singular point, and $\psi$ is a frontal but not a front on $\left\{ u^2-v^2=0 \right\} $ (they are all corank 1 map-germs except for on the origin.) while $\psi$ is a front on $C:=\left\{ 4-9(u^2-v^2)=0 \right\}$ (Figure 1).

By using the criteria for cuspidal edges and swallowtails, we can completely classify the singularities on
$C$. In fact, at any point on $C$ except for the point $(-\frac{2}{3},0)$, singularities are $\mathcal{A}$-equivalent
to the cuspidal edges while swallowtail appears at $(-\frac{2}{3},0)$. This assertion is verified as below.

Let $C_1:=C \cap \left\{u \geq 0 \right\}$ and $C_2:=C \cap \left\{u \leq 0 \right\}$. First, we show that
the singularity is $\mathcal{A}$-equivalent to the cuspidal edge at any point on $C_1$. 
$C_1$ is parametrized as $\gamma(t):= \frac{2}{3}(\cosh t, \sinh t)$. Since
\[
 \begin{cases}
  (f^2+g^2)_u(\gamma(t)) &={\frac{4}{3}}{\sinh t}(1+2{\cosh t}), \\
  (f^2+g^2)_v(\gamma(t)) &={\frac{4}{3}}({\cosh t}+1)(2{\cosh t}-1), \\
  \end{cases}
\]
the null vector field at $\gamma(t)$ is parallel to $(-({\cosh t}+1)(2{\cosh t}-1), {\sinh t}(1+2{\cosh t}))$.
Therefore
\[ \det(\gamma'(t), \eta(t)) =- ({\cosh t}+1)
\]
and it follows that the null vector field and the singular direction are transversal on $C_1$.
$C_2$ is parametrized as $\gamma(t):= -\frac{2}{3}(\cosh t, \sinh t)$ and
\[
  \begin{cases}
   (f^1-g^1)_u(\gamma(t)) &= -{\frac{4}{3}}({\cosh t}+1)(2{\cosh t}-1), \\
   (f^1-g^1)_v(\gamma(t)) &= -{\frac{4}{3}}{\sinh t}(1+2{\cosh t}), \\
  \end{cases}
\]
hold. Thus the null vector field at $\gamma(t)$ is parallel to $(-{\sinh t}(1+2{\cosh t}), ({\cosh t}+1)(2{\cosh t}-1))$.
Therefore
\[ \det(\gamma'(t), \eta(t)) = -{\sinh t}
\]
and it follows that the null vector field and the singular direction are transversal on $C_2$ except for
on the point $(-\frac{2}{3}, 0)$ and that $\det{\big |}_{t=0}(\gamma'(t), \eta(t)) \not= 0$ on $(-\frac{2}{3}, 0)$.
\end{example}

\begin{example} 
\label{ex: 34}
 Let $(F,G)=( z^3, z^4)$, then Theorem~\ref{thm:rep} also gives a concrete example of 
indefinite generalized IA-maps, which is a frontal map but not a front.

For this example, we have
 \[
 \begin{cases}
  f^1&=u^3+3uv^2, \ \ f^2=3u^2v+v^3, \\
  g^1&=u^4+6u^2v^2+v^4, \ \ g^2=4u^3v+4uv^3, \\ 
\end{cases} 
\] and this implies
\[
 {\lambda}=(f^1_u)^2-(f^2_u)^2-(g^1_u)^2+(g^2_u)^2 =(u^2-v^2)^2(9-16(u^2-v^2))
\]and
\[
\begin{cases}
 {\lambda}_u&=12u(u^2-v^2)(3-8(u^2-v^2)), \\
 {\lambda}_v&=-12v(u^2-v^2)(3-8v(u^2-v^2)). \\
\end{cases} 
\]
Therefore, singular locus is $\left\{(u,v)\in{\R}^2|(u^2-v^2)^2(9-16(u^2-v^2))=0 \right\}$. Among them, $\left\{u^2-v^2=0 \right\}$ is the set of degenerate singular point. Moreover, $\psi$ is a frontal but not front on $\left\{ u^2-v^2=0 \right\}$ (they are all corank 1 map-germs except for on the origin.) while $\psi$ is a front on $\widetilde{C}:=\left\{ 9-16(u^2-v^2)=0 \right\} $. Above calculation shows that (1,1) is the 
point where the following three conditions are satisfied, that is,
\begingroup
\renewcommand{\theenumi}{(\roman{enumi})}    
\renewcommand{\labelenumi}{(\roman{enumi})}  
\begin{enumerate}
 \item\label{item:1} $(1,1)$ is a degenerate singular point.
 \item\label{item:2} $\psi$ is a frontal but not a front of corank 1 on $(1,1)$.
 \item\label{item:3} the set of degenerate singular point around $(1,1)$ is locally a smooth curve.
\end{enumerate}
\endgroup
It is worth mentioning that any map-germs satisfying above conditions which appear in Ishikawa-Machida's formulation 
is not $\mathcal{A}$-equivalent to this example (Figure 2). We prove this fact in Section~\ref{sec:comparison}.

 As similar to the Example~\ref{ex: 23}, the singularities on $\widetilde{C}$ is completely classified 
by the same way. At any point on $\widetilde{C}$ except for the point $(-\frac{2}{3},0)$, singularity is 
$\mathcal{A}$-equivalent to the cuspidal edge while swallowtail appears at $(-\frac{2}{3},0)$. 
\end{example}

\begin{figure}
 \begin{center}\footnotesize
  \begin{tabular}{c@{\hspace{1cm}}c}
   \includegraphics[width=5.5cm]{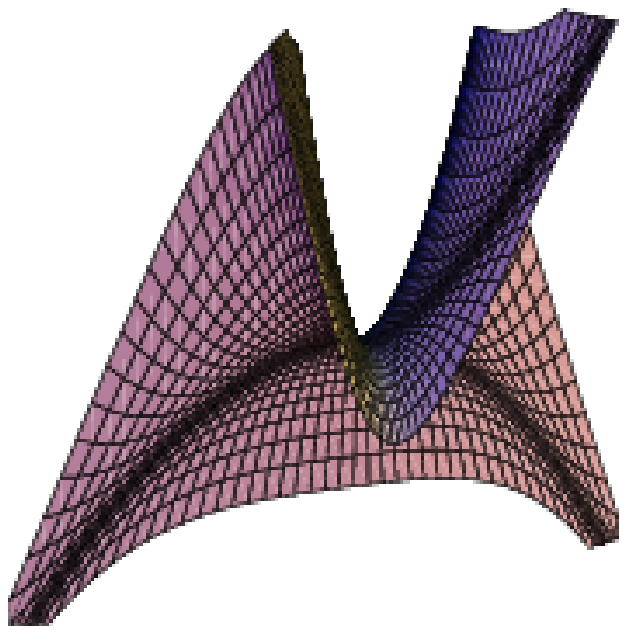}
   &
   \includegraphics[width=5.5cm]{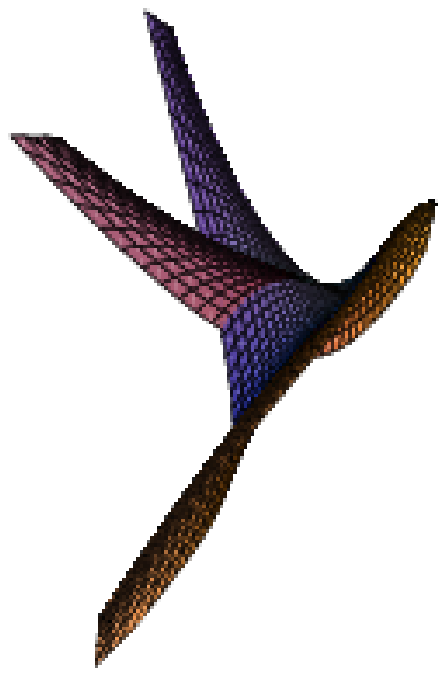}\\
     Example~\ref{ex: 23} &

    {Example~\ref{ex: 34}} \\
   &
  \end{tabular}
 \end{center}
  \caption{Example~\ref{ex: 23} and Example~\ref{ex: 34}.}
  \label{fig:lip-and-nondeg}
\end{figure}

\section{Comparison with formulation by Ishikawa-Machida}
\label{sec:comparison}
 As mentioned in Remark~\ref{rmk: IM}, Ishikawa and Machida also studied improper affine spheres with
singularities in another setting. They considered (generalized) geometric solutions of a certain Monge-Amp\`{e}re 
system $\mathcal{M}$ on ${\R}^5$ and studied the singularities of the projections of such (generalized) geometric 
solutions. Here, the projection of a (generalized) geometric solution of $\mathcal{M}$ is nothing but an improper
affine sphere outside singular points.
 The singularities of indefinite generalized IA-maps and those of the projection of 
generalized geometric solutions of $\mathcal{M}$ share some same properties as mentioned in Remark~\ref{rmk: IM}.
However, there are also different properties between singularities which appear in both formulations. 
Indeed, we can find the singularity on indefinite generalized IA-maps which does not appear 
on the generalized geometric solution of $\mathcal{M}$. 

 Before proving this, we should review the Ishikawa-Machida's formulation briefly.
 
 In \cite{Ishikawa-Machida}, they studied the singularities of graphs $z=f(x,y)$ in $xyz$-space ${\R}^3$
where $f$ is a solutions of the Monge-Amp\`{e}re type equation $f_{xx}f_{yy}-f_{xy}^2=c$. Here, 
the equation $f_{xx}f_{yy}-f_{xy}^2=c$ can be considered geometrically in terms of the differential system
$\mathcal{M}$ on $xyzpq$-space ${\R}^5$, which is generated by 
\begin{equation}
\label{eq: differential system}
  \omega= cdx{\wedge}dy-dp{\wedge}dq, \ \ \theta=dz-pdx-qdy,
\end{equation}
where $p,q$ represent $z_x=f_x, z_y=f_y$ respectively.
Here, $D=\left\{ \theta=0 \right\}$ is the contact structure on $T{\R}^5$.
For the graph $z=f(x,y)$ of a solution of the Monge-Amp\`{e}re type equation 
$f_{xx}f_{yy}-f_{xy}^2=c$ in $xyz$-space ${\R}^3$, we define its lift 
$L_f:{\R}^2 \rightarrow {\R}^5$ to ${\R}^5$ as $L_f(x,y,z):=(x,y,f(x,y),f_x(x,y),f_y(x,y))$.
Obviously, the lift of the graph of $f_{xx}f_{yy}-f_{xy}^2=c$ annihilates both $\omega$ and $\theta$,
that is, it is a Legendrian immersion into ${\R}^5$ which annihilates $\omega$.
Taking this into consideration, a \textit{geometric solution} (respectively, 
\textit{generalized geometric solution}) of $\mathcal{M}$ is defined as
a Legendrian immersion with respect to the contact structure $D$
(respectively, a map annihilating $\omega$ which is not necessarily an immersion) of ${\R}^2$ 
into ${\R}^5$, which also annihilates $\theta$.

 Among the singularities of improper affine spheres $\psi=(x,y,z):{\R}^2 \rightarrow {\R}^3$ which appear as 
the projection of (generalized) geometric solution $f=(x,y,z,p,q):{\R}^2 \rightarrow {\R}^5$ of 
the Monge-Amp\`{e}re system $\mathcal{M}$, we are especially interested in singularities of corank 1. 
In this case, by implicit function theorem, we can take a coordinate system $(u,v)$ of ${\R}^2$ around $(0.0)$ 
such that $\psi(u,v)=(u,y(u,v),z(u,v))$. Since $f^{*}\theta =0$, we have 
\begin{equation}
\label{eq:theta}
  z_u=p+qy_u \ \ \ \textrm{and} \ \ \  z_v=qy_v. 
\end{equation}
So the unit normal $\tilde{\nu}$ of $f$ becomes $\tilde{\nu}=\dfrac{1}{\sqrt{p^2+q^2+1}}(-p,-q,1)$ because $f_u=(1,y_u,z_u)$ 
and $f_v=(0,y_v,z_v)$. Hence the signed area density of $f$ is $\lambda=(1+q^2)y_v$ and we can conclude that
a point $(u_0,v_0)$ is a singular point if $y_v(u_0,v_0)=0$ and that a singular point $(u_0,v_0)$ is degenerate
(respectively, non-degenerate) if $y_{uv}(u_0,v_0)=y_{vv}(u_0,v_0)=0$ (resp. neither $y_{uv}(u_0,v_0)\not=0$ nor 
$y_{vv}(u_0,v_0)\not=0$ holds). 

Now, based on the above review, we can prove the following proposition on singularities on 
improper affine spheres.
 
\begin{proposition}
\label{prop:comparison}
 The germ of the map of Example~\ref{ex: 34} at $(1,1)\in{\R}^2$ is not $\mathcal{A}$-equivalent to 
any germ of generalized geometric solution of $\mathcal{M}$.
\end{proposition}

The proof is based on the following well-known facts about singularities.

\begin{fact}[Mather division theorem \cite{GG}]
\label{fact:Mather}
Let $F$ be a smooth real-valued function defined on a neighborhood of $0$ in ${\R}\times {\R}$ such that
$F(0,t)=g(t)t^k$ where $g(0)\not=0$ and $g$ is smooth on some neighborhood of $0$ in ${\R}$. Then given
any smooth real-valued function $G$ defined on a neighborhood of $0$ in ${\R}\times{\R}$, there exist
smooth functions $q$ and $r$ such that
\begin{enumerate}
 \item $G=qF+r$ on a neighborhood of $0$ in ${\R}\times{\R}$, and 
 \item $r(x,t)= \displaystyle \sum_{i=0}^{k-1} r_i(x)t^i$ for $(t,x)\in{\R}\times{\R}$ near $0$.
\end{enumerate}
\end{fact}

Next, we denote by $\mathcal{E}_2$ the set of smooth functions on ${\R}^2$.
For a smooth map-germ $f=(f^1, f^2, f^3):({\R}^2,0) \rightarrow ({\R}^3,0)$,
we denote by $I(f)$ the ideal of $\mathcal{E}_2$ generated by $f^1, f^2$ and $f^3$
and define $Q(f):=\mathcal{E}_2/I(f)$.

\begin{fact}[\cite{Mather}]
 Let $f,g:({\R}^2,0) \rightarrow ({\R}^3,0)$ be two map-germs from ${\R}^2$ to ${\R}^3$.
If $f$ and $g$ are $\mathcal{A}$-equivalent each other, then $Q(f)$ and $Q(g)$ are isomorphic
as $\R$-algebras.
\end{fact}

\begin{proof}[Proof of Proposition~\ref{prop:comparison}]
 The map of Example~\ref{ex: 34} is concretely expressed as
\[
  \psi(u,v)=(u^3 + 3uv^2-u^4-6u^2v^2-v^4, v^3+3u^2v+4u^3v+4uv^3, \Phi)
\]
where 
\begin{alignat*}{1}
 {\Phi} &=
         \frac{1}{2}u^6-\frac{3}{2}u^4v^2+\frac{3}{2}u^2v^4-\frac{1}{2}v^6-\frac{1}{7}u^7-3u^5v^2-5u^3v^4-uv^6 \\
        &
        \hspace{0.3in}
        \begin{aligned}
        -\frac{1}{2}u^8+2u^6v^2-3u^4v^4 + 2u^2v^6 - \frac{1}{2}v^8.
        \end{aligned}     
\end{alignat*}

 Define $\tilde{\psi}:= (\tilde{\psi}^1, \tilde{\psi}^2, \tilde{\psi}^3):=\psi(u-1, v-1)$, then 
$Q(\tilde{\psi}):={\mathcal{E}}_2/{\langle\tilde{\psi}^1, \tilde{\psi}^2, \tilde{\psi}^3  \rangle}_{\mathcal{E}_2} 
\cong {\langle 1, v, v^2 \rangle}_{\R}$.
Recall that the map-germ $(\psi,(1,1))$ (and so $(\tilde{\psi},(0,0))$) possesses 
the following property at a singular point $p\in{\Sigma}^2$. (Example~\ref{ex: 34})

\begingroup
\renewcommand{\theenumi}{(\roman{enumi})}    
\renewcommand{\labelenumi}{(\roman{enumi})}  
\begin{enumerate}
 \item\label{item:1} $p$ is a degenerate singular point.
 \item\label{item:2} $\psi$ is frontal but not front of corank 1 on $p$.
 \item\label{item:3} the set of degenerate singular point around $p$ is locally a smooth curve.
\end{enumerate}
\endgroup

In the following, we will show that for $f=(x,y,z):{\R}^2\rightarrow {\R}^3$, the projection of 
a (generalized) geometric solution of $\mathcal{M}$ and for $(u_0,v_0) \in {\R}^2$, if the map-germ
$(f,(u_0,v_0))$ is $\mathcal{A}$-equivalent to $(\tilde{\psi},(0,0))$ then contradiction occurs.
From assumption, $(f,(u_0,v_0))$ also have the above properties (i), (ii) and (iii)
because these three conditions are preserved under the same $\mathcal{A}$-equivalent class.
By the change of coordinates of ${\R}^2$ and ${\R}^3$, we can assume that $(u_0,v_0)=(0,0)$ and 
$f(0,0)=(0,0,0)$.

First, the condition (i) and (ii) implies that $x(u,v)=u$ and $y_v(0,0)=y_{uv}(0,0)=y_{vv}(0,0)=0$.
Therefore the Taylor expansion of $y(u,v)$ around $(0,0)$ becomes as the following form:
\begin{equation}
\label{eq:y}
  y(u,v)=a_kv^kq(v)+{ \sum_{i\not=0} a_{ij}u^iv^j}
\end{equation}
where $k \geq 3$, $a_k \not=0 $ and $q(v)$ is a smooth function with $q(0)\not=0$ because otherwise, 
$y(u,v),z(u,v) \in I(u)$ from ~\eqref{eq:theta}. It follows from ~\eqref{eq:theta} and ~\eqref{eq:y} 
that $z(u,v) \in I(y(u,v))$. Hence $Q(f) \cong {\langle 1, v, \cdots, v^{k-1} \rangle}_{\R}$ and 
$Q(\tilde{\psi}) \cong Q(f)$ implies that $k=3$.

By the way, the condition (iii) implies that 
\begin{quotation}
 (iv) $\left\{y_v=0 \right\} \cap \left\{y_{vv}=0\right\}$ is locally a smooth curve.
\end{quotation}
From ~\eqref{eq:y}, Fact~\ref{fact:Mather} implies that
\[
  y_v(u,v)=P(u,v)y_{vv}(u,v)+R(u)
\]
holds.
The condition (iv) implies that $y_v(u,v)=0$ if $y_{vv}(u,v)=0$ because $y_{vv}(u,v)=0$ is locally a
smooth curve and that
there exists some $v$ for any $u$ such that $y_v(u,v)$ and $y_{vv}(u,v)$ holds around $(0,0)$
unless the singular set is $\left\{u=0 \right\}$.
Thus $R(u)=0$ holds for any $u$ around 0.
Therefore, locally 
\begin{equation}
\label{eq:MA}
  y_v(u,v)=P(u.v)y_{vv}(u,v)
\end{equation}
holds.
The general solution of ~\eqref{eq:MA} is 
\[
y(u,v)=\int{\exp{\left(\int{\frac{dv}{P(u,v)}}+\alpha(u)\right)}dv}+\beta(u),
\] 
but this contradicts to the condition (iv) because $\left\{y_v=0 \right\} = \emptyset$.
\end{proof}

\begin{acknowledgements}

\end{acknowledgements}
 The author would like to thank Kotaro Yamada for reading carefully the manuscript and giving him some suggestive
comments. He also would like to thank Hitoshi Furuhata, Jun-ichi Inoguchi, Go-o Ishikawa and Takashi Kurose for valuable
suggestions.


\begin{thebibliography}{1} 
\bibitem{Blaschke}
 W. Blaschke: {\em Vorlesungen \"{u}ber Differentialgeometrie I\hspace{-.1em}I, Affine Differentialgeometrie}.
  Springer, Berlin, (1923).  

\bibitem{Cortes}
 V. Cort\'{e}s: {\em A holomorphic representation formula for parabolic hypherspheres}.
  Banach Center Publications. {\bf57}, 11-16 (2002).

\bibitem{Cortes-Lawn-Schafer}
 V. Cort\'{e}s, M.-A. Lawn, L. Sch\"{a}fer: {\em Affine hyperspheres associated to special para-K\"{a}hler manifolds}.
  preprint

\bibitem{FSUY}
 S. Fujimori, K. Saji, M. Umehara, K. Yamada: {\em Singularities of maximal surfaces}. 
  to appear in Math. Z.

\bibitem{GG}
 M. Golubitsky, V. Guillemin: {\em Stable mappings and their singularities}.
  Springer-Verlag, New York (1973)
 
\bibitem{KRSUY}
 M. Kokubu, W. Rossman, K. Saji, M. Umehara, K. Yamada: {\em Singularities of flat fronts in hyperbolic space}.
  Pacific J. Math. {\bf221} (2005), 303-351.

\bibitem{flat fronts}
 M. Kokubu, M. Umehara, K. Yamada: {\em Flat fronts in hyperbolic 3-space}.
  Pacific J. Math. {\bf216} (2004), 149-175.

\bibitem{Inoguchi-Toda}
 J. Inoguchi, M. Toda: {\em Timelike minimal surfaces via loop groups}.
  Acta Appl. Math. {\bf83} (2004), 313-355.
  
\bibitem{Ishikawa}
 G. Ishikawa: A private communication. (2007).

\bibitem{Ishikawa-Machida}
 G. Ishikawa, Y. Machida: {\em Singularities of improper affine spheres and surfaces of constant Gaussian curvature}.
  Internat. J. Math. {\bf17}, (2007), 269-293.

\bibitem{Kurose lecture}
 T. Kurose: Lecture given at Kyushu University. (2006).

\bibitem{Kurose}
 T. Kurose: A private communication. (2007).

\bibitem{Li-Simon-Zhao}
 A.-M. Li, U. Simon, Z. Zhao: {\em Global affine differential geometry of hypersurfaces}.
  Walter de Gruyter, Berlin-New York (1993)

\bibitem{IGV}
 A. Mart\'{i}nez: {\em Relatives of flat surfaces in $H^3$}.
   Proceedings of  the International Workshop on integrable systems, geometry and visualization, 
    Kyushu Univ. Fukuoka, Japan. (2005) 115-132. 
\bibitem{Martinez}
  A. Mart\'{i}nez: {\em Improper affine maps}. 
 Math. Z. {\bf 249} (2005), 755-766.

\bibitem{Mather}
 J. Mather: {\em Stability of $C^{\infty}$ mappings. I\hspace{-.1em}V: Classification of stable germs by
  R-algebras}. Publ. Math. I.H.E.S. {\bf37} (1969), 223-248.

\bibitem{Matsuura-Urakawa}
 N. Matsuura, H. Urakawa: {\em Discrete improper affine spheres}.
  J. Geom. Phys. {\bf45} (2003), 164-183. 
 
\bibitem{Nomizu-Sasaki}
  K. Nomizu, T. Sasaki: {\em Affine differential geometry}.
   Cambridge University Press. (1994)



\end{thebibliography}
\end{document}